\documentclass[12pt]{elsarticle}
\usepackage{verbatim}
\usepackage{lineno,hyperref}
\usepackage{amssymb,amsmath,amsthm, amsfonts, mathrsfs, bbm, dsfont}
\usepackage[utf8]{inputenc}

\newtheorem{theorem}{Theorem}[section]
\newtheorem{lemma}{Lemma}[section]
\newtheorem{proposition}{Proposition}[section]
\newtheorem{corollary}{Corollary}[section]

\newtheorem{definition}{Definition}[section]

\newtheorem{remark}{Remark}

\usepackage{extarrows}
\modulolinenumbers[5]
\usepackage[top=1in, bottom=1in, left=1.25in, right=1.25in]{geometry}
\usepackage{setspace}%使用间距宏包

\begin{document}
\begin{spacing}{1.1}
\begin{frontmatter}

\title{Kadec-Klee property for convergence in measure of noncommutative Orlicz spaces\tnoteref{mytitlenote}}
\tnotetext[mytitlenote]{The research has been supported by National Science Foundation of China (Grant No.10971011 ).}

%% Group authors per affiliation:
%\author{Elsevier\fnref{myfootnote}}
%\address{Radarweg 29, Amsterdam}
%\fntext[myfootnote]{Since 1880.}

%% or include affiliations in footnotes:
\author[mymainaddress,mysecondaryaddress]{Ma Zhenhua}

\ead{mazhenghua\_1981@163.com}

\author[mysecondaryaddress]{Jiang Lining \corref{Jiang Lining}}
\cortext[Jiang Lining]{Corresponding author}

\ead{jianglining@bit.edu.cn}

\author[mymainaddress]{Ji Kai}
\ead{jikai68@163.com}

\address[mymainaddress]{Hebei University of Architecture, Zhangjiakou, 075024, P. R. China}
\address[mysecondaryaddress]{Beijing
Institute of Technology, Beijing, 100081, P. R. China}

\begin{abstract}
In this paper, we study the Kadec-Klee property for convergence in measure of noncommutative Orlicz spaces $L_{\varphi}(\widetilde{\mathcal{M}},\tau)$, where $\widetilde{\mathcal{M}}$ is a von Neumann algebra, and $\varphi$ is an Orlicz function. We show that if $\varphi\in\Delta_{2}$,  $L_{\varphi}(\widetilde{\mathcal{M}},\tau)$ has the Kadec-Klee property in measure. As a corollary, the dual space and reflexivity of $L_{\varphi}(\widetilde{\mathcal{M}},\tau)$ are given.
\end{abstract}


\end{frontmatter}

\linenumbers
\section{Preliminaries}
As is well known,  the Kadec-Klee property was firstly studied by J. Radon \cite{Radon}. This property said that if $(E,\|\cdot\|_{E})$ is a normed linear space, then $E$ is said to have the Kadec-Klee property (sometimes called the Radon-Riesz property, or property (H)) if and only if sequential weak convergence on the unit sphere coincides with norm convergence. For example, in \cite{Reisz1} and \cite{Reisz2}, F. Riesz  showed that the classical $L_{p}$-spaces, $1<p<\infty$ have the Kadec-Klee property.

%if $E$ is a separable symmetric sequence space, then it was shown in \cite{BS, Ar} that $E$ has the Kadec-Klee property if and only if the associated unitary matrix space $C_{E}$ has the Kadec-Klee property. Here $C_{E}$ is the space of all compact operators on $l^{2}$ for which $s(x)\in E$ with norm given by $\|x\|_{C_{E}}=\|s(x)\|_{E}$, where $s(x)=\{s_{n}(x)\}^{\infty}_{1}$ is the sequence of $s$-numbers of $x$.

In this paper, we study the  Kadec-Klee property for convergence in measure of noncommutative Orlicz spaces $L_{\varphi}(\widetilde{\mathcal{M}},\tau)$  \cite{Chilin}. Namely, if for any $x\in L_{\varphi}(\widetilde{\mathcal{M}},\tau)$ and any sequence $(x_{n})$ in $L_{\varphi}(\widetilde{\mathcal{M}},\tau)$ such that $\|x_{n}\|\rightarrow \|x\|$ and $x_{n}\rightarrow x$ in measure, we have $\|x_{n} -x\|\rightarrow 0$ \cite{AM,Chilin}.

If $E$ is a Banach space, we define a order $``\leq" $ on $E$, then the Banach space $E$ is said to be the order continuous if for any element $x\in E$ and any sequence $(x_{n})$ in $E_{+}$ (the positive cone in $E$) with $0\leq x_{n}\leq |x|$ and $x_{n}\rightarrow0\,\,m$-a.e., there holds $\|x_{n}\|\rightarrow0$. We note that the norm $\|\cdot\|_{E}$ on the symmetric space $E$ is order continuous if and only if $E$ is separable.

As usual, $E$ is said to be lower locally uniformly monotone ($E\in (LLUM)$ for short), whenever for any $x\in E_{+}$ with $\|x\|_{E}=1$ and any $\varepsilon\in(0,1)$ there is $\delta=\delta(x,\varepsilon)\in(0,1)$ such that the conditions $0\leq y\leq x$ and $\|y\|_{E}\geq \varepsilon$ imply $\|x-y\|\leq 1-\delta$. $E$ is said to be upper locally uniformly monotone ($E\in (ULUM)$ for short), whenever for any $x\in E_{+}$ with $\|x\|_{E}=1$ and any $\varepsilon>0$ there is $\delta=\delta(x,\varepsilon)>0$ such that the conditions $ y\geq 0$ and $\|y\|_{E}\geq \varepsilon$ imply $\|x+y\|\geq 1+\delta$ \cite{Hudzik}.

It is useful to formulate the local uniform monotonicity properties sequentially. Clearly, $E\in (LLUM)$ (resp. $E\in (ULUM)$) if and only if for any $x\in E_{+}, \,\,x\neq0$, and each sequence $(x_{n})$ in $E_{+}$ such that $x_{n}\leq x$ (resp. $x\leq x_{n}$) and $\|x_{n}\|_{E}\rightarrow\|x\|_{E}$, there holds $\|x_{n}-x\|_{E}\rightarrow0$.

%%\section{Preliminaries}
Now, we collect some of the basic facts and notation that will be used in this paper. Noncommutative integration theory was first introduced by Irving Segal \cite{Segal}, and is a fundamental tool in many theories, such as operator theory and statistical model \cite{LW}. In this paper we study some aspects of the theory of noncommutative Orlicz spaces, that is, spaces of measurable operators associated to a noncommutative Orlicz functional. The theory of Orlicz spaces associated to a trace was introduced by Muratov \cite{Muratov} and Kunze \cite{Kunze} and were respectively defined by Kunze \cite{Kunze} and Rashed et al \cite{Rashed} in an algebraic way and by Sadeghi \cite{Ghadir} employing modular spaces. In this paper we take Sadeghi's approach and continue this line of investigation.

From now on, by $\mathcal{M}$ we denote a semi-finite von Neumann algebra acting on a Hilbert space $\mathcal{H}$ with a normal semi-finite faithful trace $\tau.$ The identity in $\mathcal{M}$ is denoted by $\mathbf{1}$ and we denote by $\mathcal{P}(\mathcal{M})$ the complete lattice of all self-adjoint projections in $\mathcal{M}$.
A densely-defined closed linear operator $x: \mathcal{D}(x)\rightarrow \mathcal{H}$ with domain $\mathcal{D}(x)\subseteq\mathcal{H}$ is called affiliated with $\mathcal{M}$ if and only if $u^{\ast}xu=x$ for all unitary operators $u$ belonging to the commutant $\mathcal{M^{\prime}}$ of $\mathcal{M}$. Clearly, if $x\in \mathcal{M}$ then $x$ is affiliated with $\mathcal{M}$.
If $x$ be a (densely-defined closed) operator affiliated with $\mathcal{M}$ and $x=u|x|$ be the polar decomposition, where $|x|=(x^{\ast}x)^{\frac{1}{2}}=\int^{\infty}_{0}\lambda de_{\lambda}(|x|)$ be the spectral decomposition and $u$ is a partial isometry, then $x$ said to be $\tau$-measurable if and only if there exists a number $\lambda\geq0$ such that $\tau(e_{(\lambda,\infty)}(|x|))<\infty$, where $e_{[0,\lambda]}$ is the spectral resolution of $|x|$. The collection of all $\tau$-measurable operators is denoted by $\widetilde{\mathcal{M}}$.
%With the sum and product defined as the respective closure of the algebraic sum and product, it is well known that $\widetilde{\mathcal{M}}$ is a $\ast$-algebra.
\begin{comment}
\begin{definition}(\cite{Nelson})
Given $0<\varepsilon,\delta\in \mathds{R}$, we define
\begin{eqnarray*}
 \mathcal{V}(\varepsilon,\delta)=\{x\in\widetilde{\mathcal{M}}: \,\, there \,\, exists \,\, e\in\mathcal{P}(\mathcal{M})
 \,\,such\,\,that\,\, \\e(\mathcal{H})\in\mathcal{D}(x), \|xe\|_{\mathcal{B(\mathcal{H})}}\leq\varepsilon \,\,and \,\, \tau(\mathbf{1}-e)\leq\delta\}.
\end{eqnarray*}
\end{definition}
Here, $\varepsilon,\delta$ run over all strictly positive numbers. An alternative description of this set is given by
$$\mathcal{V}(\varepsilon,\delta)=\{x\in\widetilde{\mathcal{M}}: \tau(e_{(\varepsilon,\infty)}(|x|))<\delta\}.$$
\end{comment}
We say that $\{x_{n}\}$ converges to $x$ in measure topology $(x_{n}\xrightarrow{\tau_{m}}x$ for short $)$, if $\lim_{n\rightarrow\infty}\tau(e_{(\varepsilon,\infty)}(|x_{n}-x|))=0$ for any $\varepsilon>0$ \cite{Nelson}.

In the setting of $\tau$-measurable operators, the generalized singular value functions are the analogue (and actually, generalization) of the decreasing rearrangements of functions in the classical settings. In details, for $x\in\widetilde{\mathcal{M}}$, the generalized singular value function $\mu(x): [0,\infty]\rightarrow[0,\infty]$ is defined by
$$\mu_{t}(x)=\inf\{s\geq0: \tau\left(e_{(s,\infty)}(|x|)\right)\leq t \}, \,\, t>0.$$
It is well known that the $\mu_{(\cdot)}(x)$ is a decreasing right-continuous function on the positive half-line $[0,\infty)$ \cite{Fack}.

If $x\in\widetilde{\mathcal{M}}$ and $x\geq0$, then
$$\tau(x)=\int^{\infty}_{0}\mu_{t}(x)dt$$
 and for a continuous function $\varphi$ on $[0,\infty)$ with $\varphi(0)=0$, we have \cite{Fack}
$$\tau(\varphi(|x|))=\int^{\infty}_{0}\varphi(\mu_{t}(x))dt.$$

Next we recall the definition and  some basic properties of noncommutative Orlicz spaces.

A convex function $\varphi: [0,\infty)\rightarrow[0,\infty]$ is called an Orlicz function
if it is nondecreasing and continuous for $\alpha>0$ and such that $\varphi(0)=0,\,\varphi(\alpha)>0$ and $\varphi(\alpha)\rightarrow\infty$ as $\alpha\rightarrow\infty$ \cite{Chen}. Further we say an Orlicz function $\varphi$ satisfies the $\Delta_{2}$-condition, shortly $\varphi\in\Delta_{2}$, if there exists a constant $k>0$ such that $\varphi(2u)\leq k\varphi(u)$ for all $u>0$.
Generally speaking, $\Delta_{2}$-condition plays a very important role in the theory of either classic Orlicz spaces \cite{Chen} or noncommutative classic Orlicz \cite{Ghadir,Zhenhua}. For the background of Orlicz functions and Orlicz spaces one can see \cite{Rao,Chen}.

Suppose $x\in\widetilde{\mathcal{M}}$ and $\varphi$ is an Orlicz function, if we denote $\widetilde{\rho}_{\varphi}(x)=\tau(\varphi(|x|))$, then $\tau(\varphi(|x|))$ is a convex modular \cite{Ghadir}, hence we can define a corresponding modular space which is named noncommutative Orlicz space as follows:
$$L_{\varphi}(\widetilde{\mathcal{M}},\tau)=\{x\in\widetilde{\mathcal{M}}:  \tau(\varphi(\lambda|x|))<\infty \,\, for\,\,some\,\,\lambda>0\}.$$
We equip this space with the Luxemburg norm
$$\|x\|=\inf\{\lambda>0: \tau\Big(\varphi\Big(\frac{|x|}{\lambda}\Big)\Big)\leq1\}.$$

In the case when $\varphi(x)=|x|^{p},\,\,1\leq p<\infty$ for any $\tau$-measurable operator $x\in{\widetilde{\mathcal{M}}}$, then $\varphi\in\Delta_{2}$ and $L_{\varphi}(\widetilde{\mathcal{M}},\tau)$ is nothing but the noncommutative space
$L_{p}(\widetilde{\mathcal{M}},\tau)
=\left\{x\in\widetilde{\mathcal{M}}: \tau\left(|x|^{p}\right)<\infty\right\}$ \cite{Zhenhua} and the Luxemburg norm generated by this function is expressed by the formula
$$\|x\|_{p}=\big(\tau(|x|^{p})\big)^{\frac{1}{p}}.$$

One can define another norm on  $L_{\varphi}(\widetilde{\mathcal{M}},\tau)$ as follows
$$\|x\|^{o}=\sup\{\tau(|xy|): y\in  L_{\psi}(\widetilde{\mathcal{M}},\tau)\,\, and\,\, \tau(\psi(y))\leq1 \},$$
where $\psi: [0,\infty)\rightarrow [0,\infty] $ defined by $\psi(u)=\sup\{uv-\varphi(v): v\geq 0\}$. Here we call $\psi$ the complementary function of $\varphi$.

For more information on the theory of noncommutative Orlicz spaces we refer the reader to \cite{Muratov,Rashed,Rashed1,Ghadir,Kunze,Zhenhua}.
\section{Main results}
In this section, we firstly prove that $L_{\varphi}(\widetilde{\mathcal{M}},\tau)$ have Kadec-Klee property for convergence in measure implies $\varphi\in\Delta_{2}$. And, we find that  $\varphi\in\Delta_{2}$ is necessary of this property. As  a corollary of the Theorem 2.2, $L_{\varphi}(\widetilde{\mathcal{M}},\tau)$ is order continuous, hence the K\"{o}the dual is identified the Banach dual.

\begin{theorem}
If $L_{\varphi}(\widetilde{\mathcal{M}},\tau)$ has Kadec-Klee property for convergence in measure, then $\varphi\in\Delta_{2}$.
\end{theorem}
\begin{proof}
Suppose $\varphi\notin\Delta_{2}$, we choose $\{u_{k}\}^{\infty}_{k=1}\in \widetilde{\mathcal{M}}$ and select mutually orthogonal projections $\{e_{k}\}^{\infty}_{k=1}\in\mathcal{P}(\mathcal{M}) $  in $\mathcal{M}$ with $\tau(e_{n})\rightarrow 0$ such that
$$\varphi\left(\left(1+\frac{1}{k}u_{k}\right)\right)>2^{k}\varphi(u_{k})$$ and $$\varphi(u_{k})\tau(e_{k})=\frac{1}{2^{k}},$$
where $k\in\mathbb{N}$.

Define $x=\sum\limits_{k=1}^{\infty}u_{k}e_{k}$ and $x_{n}=\sum_{k=1}^{\infty}u_{k}e_{k}-u_{n}e_{n}$
then $\|x_{n}\|\rightarrow\|x\|$ and for any $s>0$, by Lemma 2.6 of \cite{Fack} and $\tau(e_{n})\rightarrow 0$,
\begin{eqnarray*}
\tau(e_{(s, \infty)}(|x_{n}-x|))&=&\tau(e_{(s, \infty)}(|2u_{n}e_{n}|))\\
&=&\int_{0}^{\infty}\chi_{(s, \infty)}(\mu_{t}(|2u_{n}e_{n}|))dt\\
&\rightarrow&0
\end{eqnarray*}
Hence, $x_{n}\xrightarrow{\tau_{m}}x$.

On the other side, by (ii) of Proposition 3.4 in \cite{Ghadir} and  Remark 3.3 in \cite{Fack},
\begin{eqnarray*}
1\geq\tau\left(\varphi\left(\frac{e_{k}}{\|e_{k}\|}\right)\right)&=&\int_{0}^{\infty}\varphi\left(\mu_{t}\left(\frac{e_{k}}{\|e_{k}\|}\right)\right)dt\\
&=&\varphi\left(\frac{1}{\|e_{k}\|}\right)\tau(e_{k})\\
&=&\varphi\left(\frac{1}{\|e_{k}\|}\right)\frac{1}{2^{k}\varphi(u_{k})}\\
&>&\varphi\left(\frac{1}{\|e_{k}\|}\right)\frac{1}{\varphi\left(\left(1+\frac{1}{k}\right)u_{k}\right)}
\end{eqnarray*}
since $\mu_{t}(e_{k})=\chi_{[0, \tau(e_{k}))}(t)$ for any $k\in N$.

Then we have
$\|e_{k}\|>\left[\left(1+\frac{1}{k}\right)u_{k}\right]^{-1}$ and
$$\|x-x_{n}\|=\|2u_{n}e_{n}\|>2\left(1+\frac{1}{n}\right)^{-1}$$
which completes the proof.
\end{proof}
In order to prove that $L_{\varphi}(\widetilde{\mathcal{M}},\tau)$ has the Kadec-Klee property for convergence in measure, we need the following two Lemmas.
\begin{lemma}
If $\varphi\in\Delta_{2}$,  then for any sequence $\{x_{n}\}$ in $L_{\varphi}(\widetilde{\mathcal{M}},\tau)$, we have $\|x_{n}\|\rightarrow \|x\|$ if and only if $\tau(\varphi(x_{n}))\rightarrow \tau(\varphi(x))$.
\end{lemma}
\begin{proof}

Without loss of generality, suppose that $\|x\|=1$.

If $\tau(\varphi(x_{n}))\rightarrow1$, since $\tau(\varphi(x))\leq\|x\|$ if $\tau(\varphi(x))\leq1$ and $\|x\|\leq\tau(\varphi(x))$ if $\tau(\varphi(x))>1$ by Proposition 3.4 in \cite{Ghadir}. Therefore $\left|\|x_{n}\|-1\right|\leq\left|\tau(\varphi(x_{n}))-1\right|$ which implies $\|x_{n}\|\rightarrow 1$, since $\tau(\varphi(x_{n}))\rightarrow1$.

Now, assuming that $\|x_{n}\|\rightarrow1$, we fiestly need to consider two cases:

Case 1: If $\|x_{n}\|\uparrow1$ and the result is not true, then suppose that there exists an $\varepsilon_{0}>0$ and $\{x_{n}\}\subset L_{\varphi}(\widetilde{\mathcal{M}},\tau)$ such that $\tau(\varphi(|x_{n}|))\leq1-\varepsilon_{0}$. Assume that $\|x_{n}\|\geq\frac{1}{2}$ for all $n\in\mathbb{N}$. Set $a_{n}=\frac{1}{\|x_{n}\|}-1$, then $a_{n}\leq1$ for any $n\in \mathbb{N}$ and $a_{n}\downarrow0$ as $n\rightarrow\infty.$ Since $\varphi\in\Delta_{2}$, then $\sup_{n}\{\tau(\varphi(2|x_{n}|))\}<\infty,$ from (iii) of Theorem 4.4 in \cite{Fack} we have
\begin{eqnarray*}
1&=&\tau\left(\varphi\left(\frac{|x_{n}|}{\|x_{n}\|}\right)\right)\\
&=&\tau\left(\varphi\left(a_{n}|2x_{n}|+(1-a_{n})|x_{n}|\right)\right)\\
&=&\int_{0}^{\infty}\varphi\left(\mu_{t}\left(a_{n}|2x_{n}|+(1-a_{n})|x_{n}|\right)\right) \mathrm{d}t\\
&\leq&\int_{0}^{\infty}\varphi\left(\mu_{t}\left(a_{n}|2x_{n}|\right)+\mu_{t}\left((1-a_{n})|x_{n}|\right)\right) \mathrm{d}t\\
&=&\int_{0}^{\infty}\varphi\left(a_{n}\mu_{t}(2|x_{n}|)+(1-a_{n})\mu_{t}(|x_{n}|)\right) \mathrm{d}t\\
&\leq&\int_{0}^{\infty}\left(a_{n}\varphi\left(\mu_{t}(|2x_{n}|)\right)+(1-a_{n})\varphi\left(\mu_{t}(|x_{n}|)\right)\right) \mathrm{d}t\\
&=&a_{n}\tau\left(\varphi(|2 x_{n}|)\right)+(1-a_{n})\tau\left(\varphi(|x_{n}|)\right)\\
&\leq& a_{n}\sup_{n}\{\tau\left(\varphi(2|x_{n}|)\right)\}+(1-a_{n})(1-\varepsilon_{0})\\
&\rightarrow&1-\varepsilon_{0}<1.
\end{eqnarray*}
This is a contradiction and thus finishes the proof.

Case 2: If $\|x_{n}\|\downarrow1$ and the conclusion does not hold, then there exists a $\varepsilon_{0}>0$ and $\{x_{n}\}\subset L_{\varphi}(\widetilde{\mathcal{M}},\tau)$ such that $\tau(\varphi(|x_{n}|))\geq1+\varepsilon_{0}$. Assume that $\|x_{n}\|\leq2$ for $n\in \mathbb{N}$. Since $\varphi\in\Delta_{2}$, there exists a constant $L>0$ such that $\tau(\varphi(2|x_{n}|))\leq L$ for all $n\in \mathbb{N}$. By the assumption we have $0\leq1-\frac{1}{\|x_{n}\|}\leq1$ and $0\leq2-\|x_{n}\|\leq1$. Set $a_{n}=1-\frac{1}{\|x_{n}\|}, b_{n}=2-\|x_{n}\|$, then
$$0\leq a_{n}+b_{n}=(1-\frac{1}{\|x_{n}\|})+(2-\|x_{n}\|)=3-\Big(\frac{1}{\|x_{n}\|}+\|x_{n}\|\Big)\leq 1$$ for any $n\in \mathbb{N}$.

Therefore, by convexity of $\varphi$ and (iii) of Theorem 4.4 in \cite{Fack} we have
\begin{eqnarray*}
1+\varepsilon_{0}&\leq&\tau(\varphi(|x_{n}|))\\
&=&\tau\left(\varphi\left(a_{n}|2x_{n}|+b_{n}\frac{|x_{n}|}{\|x_{n}\|}\right)\right)\\
&=&\int_{0}^{\infty}\varphi\left(\mu_{t}\left(a_{n}|2x_{n}|+b_{n}\frac{|x_{n}|}{\|x_{n}\|}\right)\right) \mathrm{d}t\\
&\leq&\int_{0}^{\infty}\varphi\left(\mu_{t}\left(a_{n}|2x_{n}|\right)+\mu_{t}\left(b_{n}\frac{|x_{n}|}{\|x_{n}\|}\right)\right) \mathrm{d}t\\
&=&\int_{0}^{\infty}\varphi\left(a_{n}\mu_{t}\left(|2x_{n}|\right)+b_{n}\mu_{t}\left(\frac{|x_{n}|}{\|x_{n}\|}\right)\right)
\mathrm{d}t\\
&\leq&a_{n}\int_{0}^{\infty}\varphi\left(\mu_{t}\left(|2x_{n}|\right)\right)\mathrm{d}t+b_{n}\int_{0}^{\infty}\varphi\left(\mu_{t}\left(\frac{|x_{n}|}{\|x_{n}\|}\right)\right) \mathrm{d}t\\
&=&a_{n}\tau\left(\varphi\left(|2x_{n}|\right)\right)+b_{n}\tau\left(\varphi\left(\frac{|x_{n}|}{\|x_{n}\|}\right)\right)\\
&\leq& a_{n}L+b_{n}\\
&=&\left(1-\frac{1}{\|x_{n}\|}\right)L+(2-\|x_{n}\|)\\
&\rightarrow&1,
\end{eqnarray*}
since $\tau\left(\varphi\left(\frac{|x_{n}|}{\|x_{n}\|}\right)\right)=1$ for any $n\in \mathbb{N}$ and $1-\frac{1}{\|x_{n}\|}\rightarrow0$,
a contradiction which finishes the proof.

Now, if $\|x_{n}\|\rightarrow1$ and the conclusion does not hold, then there exists a subsequence $\{x_{n_{j}}\}\subseteq \{x_{n}\}$ which either $\|x_{n_{j}}\|\uparrow1$ or $\|x_{n_{j}}\|\downarrow1$, by Case1 or Case 2 we can get a contradiction which can get the conclusion.
\end{proof}
Using Lemma 3.4 of \cite{Fack}, we can easily get the following Lemma,
\begin{lemma}
Suppose  $\varphi\in\Delta_{2}$  and $x\in L_{\varphi}(\widetilde{\mathcal{M}},\tau)$. For any sequence $\{x_{n}\}$ in $L_{\varphi}(\widetilde{\mathcal{M}},\tau)$ such that $x_{n}\xrightarrow{\tau_{m}}x$, if the maps $s\rightarrow \mu_{s}(x)$ is continuous at $s=t$, we have that $\varphi(\mu_{t}(x_{n}))\rightarrow\varphi(\mu_{t}(x))$ or $\mu_{t}(\varphi(x_{n}))\rightarrow\mu_{t}(\varphi(x))$.
\end{lemma}

The following theorem shows that under the condition $\varphi\in\Delta_{2}$, $L_{\varphi}(\widetilde{\mathcal{M}},\tau)$ has the Kadec-Klee property for convergence in measure.
\begin{theorem}
If $\varphi\in\Delta_{2}$, let $x_{n},\,(n=1,2,\ldots)$ and $x$ belong to $L_{\varphi}(\widetilde{\mathcal{M}},\tau)$. The following two conditions are equivalent:

$(1)\,\, \lim\limits_{n\rightarrow\infty}\|x_{n}-x\|\rightarrow0,$

$(2)\,\, \lim\limits_{n\rightarrow\infty}\|x_{n}\|\rightarrow\|x\|$ and $x_{n}\xrightarrow{\tau_{m}}x$.

That to say $L_{\varphi}(\widetilde{\mathcal{M}},\tau)$ has the Kadec-Klee property for convergence in measure when $\varphi\in\Delta_{2}$.
\end{theorem}
\begin{proof}
$(1)\Rightarrow(2)$: If (1) is true, first we note that $|\|x_{n}\|-\|x\||\leq \|x_{n}-x\|$, hence $\|x_{n}\|\rightarrow\|x\|$.

Secondly, by (iii) of Proposition of \cite{Ghadir}, $\varphi\in\Delta_{2}$ implies $\|x_{n}-x\|\rightarrow0\Leftrightarrow \tau(\varphi(|x_{n}-x|))=\int^{\infty}_{0}\varphi(\mu_{t}(x_{n}-x))dt\rightarrow0$ .

If $x_{n}\nrightarrow x$ in  measure, then for any $\varepsilon>0$ there exists $t_{0}>0$ and $k_{0}\in \mathbb{N}$ such that $$\varphi(\mu_{t}(x_{n}-x))=\mu_{t}(\varphi(x_{n}-x))>\varepsilon,$$
for any $t\in[0,t_{0})$ and any $n>k_{0}$. Denote $e=e_{[0, t_{0})}(\varphi(|x_{n}-x|))$, then $\tau(e)=\int^{\infty}_{0}\chi_{[0, t_{0})}(\mu_{t}(\varphi(|x_{n}-x)|))dt\leq t_{0},$ by (iv) (vi) of Lemma 2.5 and Lemma 4.1 in \cite{Fack},
\begin{eqnarray*}
\tau(\varphi(|x_{n}-x)|))&=&\int^{\infty}_{0}\mu_{t}(\varphi(|x_{n}-x)|))dt\\
&\geq&\sup{\tau(e\varphi(|x_{n}-x)|)e)}\\
&=&\int^{t_{0}}_{0}\mu_{t}(\varphi(|x_{n}-x)|))dt\\
&>&\varepsilon t_{0}
\end{eqnarray*}
this contradicts with (1).

$(2)\Rightarrow(1)$: By the convexity of $\varphi$  and (v),(vi) of Lemma 2.5 in \cite{Fack}, we have
\begin{eqnarray*}
0\leq\varphi\left(\mu_{t}\left(\frac{|x-x_{n}|}{2}\right)\right)&=&\varphi\left(\frac{1}{2}\mu_{t}(|x-x_{n}|)\right)\\
&\leq&\varphi\left(\frac{1}{2}\left(\mu_{t}\left(u|x|u^{\ast}+v|x_{n}|v^{\ast}\right)\right)\right)\\
&=&\varphi\left(\mu_{t}\left(\frac{u}{\sqrt{2}}|x|\frac{u^{\ast}}{\sqrt{2}}+\frac{v}{\sqrt{2}}|x_{n}|\frac{v^{\ast}}{\sqrt{2}}\right)\right)\\
&\leq&\varphi\left(\mu_{\frac{t}{2}}\left(\frac{u}{\sqrt{2}}|x|\frac{u^{\ast}}{\sqrt{2}}\right)+\mu_{\frac{t}{2}}\left(\frac{v}{\sqrt{2}}|x_{n}|\frac{v^{\ast}}{\sqrt{2}}\right)\right)\\
&\leq&\varphi\left(\frac{1}{2}\mu_{\frac{t}{2}}\left(|x|\right)+\frac{1}{2}\mu_{\frac{t}{2}}\left(|x_{n}|\right)\right]\\
&\leq&\frac{1}{2}\left[\varphi\left(\mu_{\frac{t}{2}}\left(|x|\right)\right)+\varphi\left(\mu_{\frac{t}{2}}\left(|x_{n}|\right)\right)\right].
\end{eqnarray*}

If $x_{n}\xrightarrow{\tau_{m}}x$, it follows from Lemma 3.1 of \cite{Fack} that $\lim\limits_{n\rightarrow\infty}\mu_{t}(x_{n}-x)=0$ for each $t>0$,
and by Lemma 2.1, suppose that $\tau(\varphi(|x_{n}|))\rightarrow \tau(\varphi(|x|))$,
the Fatou's Lemma and Lemma 2.2 imply that
\begin{eqnarray*}
0\leq\int^{\infty}_{0}\varphi\left(\mu_{\frac{t}{2}}\left(|x|\right)\right) \mathrm{d}t
&=&\int^{\infty}_{0}\lim_{n\rightarrow\infty} \Big [\frac{\varphi\left(\mu_{\frac{t}{2}}\left(|x|\right)\right)+\varphi\left(\mu_{\frac{t}{2}}\left(|x_{n}|\right)\right)}{2}\\
&&-\varphi\left(\mu_{t}\left(\frac{|x-x_{n}|}{2}\right)\right)\Big] \mathrm{d}t\\
&\leq&\varliminf_{n\rightarrow\infty}\int^{\infty}_{0}\Big[\frac{\varphi\left(\mu_{\frac{t}{2}}\left(|x|\right)\right)+\varphi\left(\mu_{\frac{t}{2}}\left(|x_{n}|\right)\right)}{2}\\
&&-\varphi\left(\mu_{t}\left(\frac{|x-x_{n}|}{2}\right)\right)\Big] \mathrm{d}t\\
&=&\int^{\infty}_{0}\varphi\left(\mu_{\frac{t}{2}}\left(|x|\right)\right) \mathrm{d}t-\varlimsup_{n\rightarrow\infty}\tau\left(\varphi\left(\frac{|x-x_{n}|}{2}\right)\right).
\end{eqnarray*}
Then we obtain $$-\varlimsup_{n\rightarrow\infty}\sup\tau\left(\varphi\left(\frac{|x-x_{n}|}{2}\right)\right)\geq0,$$
which implies $\tau\left(\varphi\left(\frac{|x_{n}-x|}{2}\right)\right)\rightarrow0$.

Hence $\|x_{n}-x\|\rightarrow0$ since $\varphi\in\Delta_{2}$. This completes the proof.
\end{proof}
As an application, using Theorem 2.2 we can get the following Corollary which was firstly proved in \cite{Fack}.

\begin{corollary}
 Let $x_{n}$ and $x$ be element in $L_{p}(\widetilde{\mathcal{M}},\tau) \ \ (1<p<\infty)$. Then the following two conditions are equivalent:

$(1)\,\, \lim\limits_{n\rightarrow\infty}\|x_{n}-x\|_{p}\rightarrow0,$

$(2)\,\, \lim\limits_{n\rightarrow\infty}\|x_{n}\|_{p}\rightarrow\|x\|_{p}$ and $x_{n}\xrightarrow{\tau_{m}}x$.
\end{corollary}
In other words, the space $L_{p}(\widetilde{\mathcal{M}},\tau)$ has the Kadec-Klee property for convergence in measure.

%A Banach lattice $E$ with a partial order $\leq$ is strictly monotone ($E\in (SM)$ for short) if the conditions $0\leq y\leq x\in E$ and $y\neq x$ imply that $\|y\|_{E}<\|x\|_{E}$.

%\section{The Banach dual space of $L_{\varphi}(\widetilde{\mathcal{M}},\tau)$}
Combined with the Theorem 2.2 and Lemma 2.1 we can get
\begin{corollary}
If $\varphi\in\Delta_{2}$, let $x_{n},\,(n=1,2,\ldots)$ and $x$ belong to $L_{\varphi}(\widetilde{\mathcal{M}},\tau)$ with $x_{n}\xrightarrow{\tau_{m}}x$, then

$(1)$ The noncommutative Orlicz spaces $L_{\varphi}(\widetilde{\mathcal{M}},\tau)$ has the property  $LLUM$.

$(2)$ The noncommutative Orlicz spaces $L_{\varphi}(\widetilde{\mathcal{M}},\tau)$ has the property  $ULUM$.
\end{corollary}

From Lemma 3.1 of \cite{Fack} and Lemma 2.1 we have
\begin{theorem}
Suppose that  $\varphi\in\Delta_{2}$. The noncommutative Orlicz spaces $L_{\varphi}(\widetilde{\mathcal{M}},\tau)$   is order continuous. Hence, it  is separable. Especially, $L_{p}(\widetilde{\mathcal{M}},\tau)$ is separable, where $1<p<\infty$.

%$(2)$ The noncommutative Orlicz spaces $L^{\varphi}(\widetilde{\mathcal{M}},\tau)$ has the property  $LLUM$.

%$(3)$ The noncommutative Orlicz spaces $L^{\varphi}(\widetilde{\mathcal{M}},\tau)$ has the property  $ULUM$.
\end{theorem}
Next, we consider the dual space of $L^{\varphi}(\widetilde{\mathcal{M}},\tau)$.
By $L^{0}(\mathbb{R}^{+},m)$ we denote the space of all $\mathbb{C}$-valued Lebesgue measurable function of $\mathbb{R}^{+}$. A Banach space $(E, \|\cdot\|_{E})$, where $E\subseteq L^{0}(\mathbb{R}^{+},m)$, is called the rearrangement-invariant Banach function space if it follows from $f\in E, g\in L^{0}(\mathbb{R}^{+},m)$ and $\mu(g)\leq \mu(f)$ that $g\in E$ and $\|g\|_{E}\leq\|f\|_{E}$. Furthermore, $E, \|\cdot\|_{E}$ is called a symmetric Banach function space if it has the additional property, that $f, g\in E$ and $g\prec\prec f$ imply that $\|g\|_{E}\leq\|f\|_{E}$. Here $g\prec\prec f$ denotes for all $t>0$:
$$\int_{0}^{t}\mu_{s}(g)ds\leq \int_{0}^{t}\mu_{s}(f)ds.$$

If the Banach space $E\subseteq\widetilde{\mathcal{M}}$ is properly symmetric then the K\"{o}the dual $E^{\times}$ is defined by setting
$$E^{\times}=\{y\in\widetilde{\mathcal{M}}: xy\in L_{1}(\mathcal{M})\ \ for \ \ all \ \ x\in E\}$$
and if $x\in\widetilde{\mathcal{M}}$, we define
$$\|x\|_{E^{\times}}=\sup\{\tau(|xy|): y\in E, \|y\|_{E}\leq1\}.$$

Next theorem shows the dual space of the $L_{\varphi}(\widetilde{\mathcal{M}},\tau)$.
\begin{theorem} If $\varphi\in \Delta_{2}$, we have
$$L_{\varphi}(\widetilde{\mathcal{M}},\tau)^{\ast}=L_{\psi}^{o}(\widetilde{\mathcal{M}},\tau),$$
where $L_{\psi}^{o}(\widetilde{\mathcal{M}},\tau)=(L_{\psi}(\widetilde{\mathcal{M}},\tau), \|\cdot\|^{o})$.
\end{theorem}
\begin{proof}
%It follows from standard argument of \cite{ZA} that the norm on $E(\mathbb{R}^{+})$ is order continuous if and only if $E^{\ast}=E^{\times}$. Suppose now that $E(\mathbb{R}^{+})$ is a rearrangement invariant symmetric Banach function space on $\mathbb{R}^{+}$. If the norm on $E(\mathbb{R}^{+})$ is  order continuous, then it follows from Proposition 3.6 of \cite{Dodds} that the norm on $E(\mathcal{M})$ is also order continuous and
Theorem 5.11 combined with Theorem 5.6 of \cite{Dodds} show that for a rearrangement invariant symmetric Banach function space $E(\mathcal{M})$, if it is order continuous, then Banach dual $E(\mathcal{M})^{\ast}$ may be identified with the space $E^{\times}(\mathcal{M})$ if $E(\mathcal{M})$. Hence, by Theorem 2.3, we can get the conclusion.
\end{proof}
Similar to the classic case, using Theorem 2.4 we can get
\begin{corollary} 
$L_{\varphi}(\widetilde{\mathcal{M}},\tau)$ is reflexive if and only if both $\varphi\in \Delta_{2}$ and $\psi\in \Delta_{2}$.
\end{corollary}
It easy to know that if $\varphi(x)=|x|^{p}\,\, (1<p<\infty)$, then $\psi(x)=|x|^{q}$ is complementary function of $\varphi$, where $\frac{1}{p}+\frac{1}{q}=1$.
 Hence, as an special example of Theorem 2.4, one have

$(1)$ $L_{p}(\widetilde{\mathcal{M}},\tau)^{\ast}=L_{q}(\widetilde{\mathcal{M}},\tau)$, where $1<p<\infty$ and $\frac{1}{p}+\frac{1}{q}=1$;

$(2)$ $L_{p}(\widetilde{\mathcal{M}},\tau)$ is reflexive when  $1<p<\infty$.

Especially, if $p=1$ then $L_{1}(\widetilde{\mathcal{M}},\tau)^{\ast}=L_{\infty}(\widetilde{\mathcal{M}},\tau)=\mathcal{M}$, but $L_{1}(\widetilde{\mathcal{M}},\tau)$ is nonreflexive since $\mathcal{M}^{\ast}\neq L_{1}(\widetilde{\mathcal{M}},\tau)$.
\section*{References}

\bibliography{mybibfile}

\begin{thebibliography}{00}

\bibitem{Akemann} C. Akemann, J. Anderson, and G. Pedersen, Triangle inequalities in operator algebras, Linear and Multilinear Algebra, 11 (1982), 167-178.
\bibitem{Rashed} M. H. A. Al-Rashed, B. Zegarlinski, Noncommutative Orlicz spaces associated to a state, Studia Math., 180 (2007), 199-209.
\bibitem{Rashed1} M. H. A. Al-Rashed, B. Zegarlinski, Noncommutative Orlicz spaces associated to a state II, Linear Algebra and its Applications., 435 (2011), 2999-3013.
\bibitem{Ar} J. Arazy, More on cnvergence in unitary matrix spaces, Proc. Amer. Math. Soc. 83 (1981), 44-48.
\bibitem{Bennet} G. Bennet, R. Sharpley, Interpolation of operators, Academic Press, London, 1988.
\bibitem{Chen} S. T. Chen, Geometry of Orlicz spaces, in: Dissertations Mathematicae, Warszawa, 1996.
\bibitem{Chilin} V. I. Chilin, P. G. Dodds, A. A. Sedaev, F. A. Sukochev, Characterizations of Kadec-Klee properties in symmetric spaces of measurable functions, Trans. Amer. Math. Soc. 348(12)(1996) 4895–4918.
\bibitem{Dodds} P. G. Dodds, T. K. Dodds, B. de Pagter, Noncommutative k\"{o}the duality, Tran. Amer. Math. Soc. 339 (1993), 717-750.
\bibitem{Fack} Th. Fack, H. Kosaki, Generalized s-numbers of $\tau$-measuable oprators, Pacific J. Math.123 (1986), 269-300.
\bibitem{Ghadir} S. Ghadir, Non-commutative Orlicz spaces associated to a modular on $\tau$-measuable operators, J. Math. Anal. Appl. 395 (2012) 705-715.
\bibitem{Hudzik} H. Hudzik, W. Kurc, Monotonicity properties of Musielak-Orlicz spaces and dominated best approximation in Banach lattices, J. Approximation Theory, 95 (1998), 353-368.
\bibitem{Kosaki} H. Kosaki, On the continuity of the map $\varphi\rightarrow|\varphi|$ from the predule of a $W^{\ast}$-algebra, Linear Multilinear Algebra 11 (1982), 167-178.
\bibitem{Kunze} W. Kunze, Noncommutative Orlicz spaces and generalized Arens algebras, Math. Nachr. 147 (1990), 123-138.
\bibitem{LW} L. E. Labuschagne, W. A. Majewski, Maps on noncomutative orlicz spaces, arXiv: 0902. 3078v2, 11 Jul 2009.
 \bibitem{Zhenhua} Z. H. Ma, L. N. Jiang, Closed subspaces and some basic topological properties of noncommutative Orlicz spaces,  to appear in Proc. Math. Sci..
\bibitem{AM} A. Medzhitov, P. Sukochev, The property (H) in Orlicz spaces, Bull. Pol. Acad. Sci. Math. 40 (1992), 5-11.
\bibitem{Muratov} M. Muratov, Noncommutative Orlicz spaces, Dokl. Akad. Nauk UzSSR 6 (1978), 11-13.
\bibitem{Musielak} J. Musielak, Orlicz spaces and modular spaces, Lecture Notes in Math. 1034 (Springer Verlag)(1983).
\bibitem{Nelson} E. Nelson, Notes on non-commutative integration, J. Funct. Anal., 15 (1974), 103-116.
\bibitem{Radon} J. Radon, Theorie und Anwendungen der absolut additiven Mengenfunctionen, Sitz. Akad. Wiss. Wien 122 (1913), 1295-1438.
\bibitem{Rao} M. M. Rao, Z. D. Ren, Theory of Orlicz spaces (New York, Basel, Hong Kong: Marcel Dekker Inc.) (1981).
\bibitem{Reisz1} F. Reisz, Sur la convergence en moyenne I, Acta Sci.Math. 4(1928/29), 58-64.
\bibitem{Reisz2} F. Reisz, Sur la convergence en moyenne II, Acta Sci.Math. 4(1928/29), 182-185.
\bibitem{Segal} I. E. Segal, A non-commutative extension of abstract integration, Ann. of Math. 57 (1953), 401-457; correction 58(1953), 595-596.
\bibitem{BS} B. Simon, Convergence in trace ideals, Proc. Amer. Math. Soc. 83 (1981), 39-43.
\bibitem{Stein} E. Stein, G. Weiss, Introduction to fourier analysis on euclidean spaces, Princeton Univ. Press, 1971.
\bibitem{Terp} M. Terp, $L^{p}$-spaces associated with von Neumann algebras, Copenhagen Univ., 1981.
\bibitem{Thompson} R. C. Thompson, Convex and concave functions of singular values of matrix sums, Pacific J. Math., 66 (1976), 285-290.
%\bibitem{Xu} Q. Xu, Embedding of $C_{q}$ and $R_{q}$ into noncommutative $L^{p}$-spaces, $1\leq p<q\leq2$, Math. Ann. 335 (2006) 109-131.
%\bibitem{Luxemburg} W. A. Luxemburg, Banach function spaces, Thesis (Delft)(1955).
\bibitem{ZA} A. C. Zaanen, Riesz spaces, II, North-Holland, Amsterdam, 1983.
\end{thebibliography}

\end{spacing}
\end{document}